\documentclass[11pt]{amsart}

\usepackage{eucal}
\usepackage{upgreek}
\usepackage{pdfsync}
\usepackage[all, cmtip]{xy}
\usepackage{amsfonts}
\usepackage{amsmath}
\usepackage{amssymb}
\usepackage{amscd}
\usepackage{color,xcolor}
\usepackage{amsthm}
\usepackage{amsxtra}
\usepackage{bbm}
\usepackage{graphicx}
\usepackage[hmargin=3cm,vmargin=3cm]{geometry}
\usepackage{mathrsfs}
\usepackage[numbers,sort]{natbib}
\usepackage{isomath}
\usepackage{enumitem}
\usepackage{latexsym}

\newtheorem{thm}{Theorem}

\newtheorem{prop}{Proposition}%[section]

\newtheorem{cor}[prop]{Corollary}

\theoremstyle{definition}

\newtheorem{df}[prop]{Definition} %[subsection]

\theoremstyle{remark}

\newtheorem{rmk}[prop]{Remark} %[subsection]
\def\mrm#1{{\mathrm{#1}}}

\def\cl#1{{\mathcal{#1}}}
\def\ul#1{{\underline{#1}}}

\newcommand{\R}{{\mathbb{R}}}
\newcommand{\Z}{{\mathbb{Z}}}
\newcommand{\C}{{\mathbb{C}}}

\newcommand{\bK}{{\mathbb{K}}}
\newcommand{\HH}{{\mathbb{H}}}

\newcommand{\bF}{{\mathbb{F}}}

\newcommand{\bs}{\bigskip}

\newcommand{\del}{\partial}

\newcommand{\sm}[1]{C^\infty(#1)}

\newcommand{\cL}{\mathcal{L}}

\newcommand{\til}[1]{\widetilde{#1}}

\newcommand{\om}{\omega}

\newcommand{\cA}{\mathcal{A}}

\newcommand{\cD}{\mathcal{D}}

\newcommand{\cP}{\mathcal{P}}

\DeclareMathOperator{\id}{\mathrm{id}}

\DeclareMathOperator{\supp}{\mathrm{supp}}

\DeclareMathOperator{\Ham}{\mathrm{Ham}}

\DeclareMathOperator{\ima}{\mathrm{im}}

\def\H2{H^{(2)}}

\usepackage[hyperindex]{hyperref}

\newcommand{\esemail}{shelukhin@dms.umontreal.ca}

\begin{document}
	
	% and $C^0$ symplectic topology

%\title{On a conjecture of Viterbo}
%\title{Flux, covers, and Viterbo conjecture for $T^n.$}

\title{Symplectic cohomology and a conjecture of Viterbo}

%\date{\today}

%\author{Asaf Kislev}
%\address{Asaf Kislev, School of Mathematical Sciences, Tel Aviv University, Israel}
%\email{\akemail}

\author{Egor Shelukhin}
\address{Egor Shelukhin, Department of Mathematics and Statistics,
	University of Montreal, C.P. 6128 Succ.  Centre-Ville Montreal, QC
	H3C 3J7, Canada}
\email{\esemail}

\bibliographystyle{abbrv}

\begin{abstract}
	
%In 2007 Viterbo has conjectured 	

We identify a new class of closed smooth manifolds for which there exists a uniform bound on the Lagrangian spectral norm of Hamiltonian deformations of the zero section in a unit cotangent disk bundle, settling a well-known conjecture of Viterbo from 2007 as the special case of $T^n.$ This class of manifolds is defined in topological terms involving the Chas-Sullivan algebra and the BV-operator on the homology of the free loop space, contains spheres and is closed under products. We discuss generalizations and various applications. 
	
%This settles a conjecture of Viterbo from 2007 as the special case of $T^n.$	
%\red{references everywhere: for math, for preliminaries, and for history\\

%\red{1. SH, HF, with filtrations \\
%2. spectral norm - more details or references?\\
%3. TQFT and actions\\
%4. maximum principle in TQFT\\
%5. A little argument in Theorem D.\\
%6. Improve $\gamma_{pt}$ to $\gamma_{alg}.$
%}
%4. Theorem D - more (or less?) details\\}

%\red{Theorem D}\\

%\red{Theorem C?}\\
	
%	We prove a conjecture of Viterbo from 2007 on the existence of a uniform bound on the Lagrangian spectral norm of Hamiltonian deformations of the zero section in unit cotangent disk bundles of the torus $T^n.$ Our proof uses elements of Abouzaid's version of Fukaya's trick to study Floer homology under Lagrangian deformations with flux, Viterbo's generating function approach to study the behavior of suitable barcodes under covering maps, as well as the author's decompactification method to pass from bounds on boundary depth to bounds on the spectral norm.
\end{abstract}

%We discuss a generalization and give applications.


\maketitle

\tableofcontents

\section{Introduction}\label{sec:intro}

In this paper we prove a well-known conjecture of Viterbo from $2007$ \cite[Conjecture 1]{Viterbo-homog} on a uniform bound on the spectral norm $\gamma(L',L)$ of Lagrangian submanifolds, Hamiltonian isotopic to the zero section, in the unit disk cotangent bundle $D^*L$ of the standard torus $L = T^n,$ taken with respect to a fixed Riemannian metric $g.$ This is a far-reaching non-linear generalization of the elementary estimate \[\gamma(df,L) = \max f - \min f \leq \left|\left|df\right|\right| \cdot \mrm{diam}(L,g)\] for each function $f \in \sm{L,\R},$ that despite considerable interest \cite{Viterbo-homog, MonznerVicheryZapolsky,MonznerZapolsky,SeyfaddiniC0Limits,HLS-coisotropic,S-Zoll,Kha-diam,BC-private} was completely open for all $n>1,$ after the recent work \cite{S-Zoll} establishing the case $n=1$. In fact, our methods yield a stronger and more general result, as they work for a large class of closed connected manifolds $L,$ and apply to arbitrary exact Lagrangian submanifolds in $D^*L.$ We remark that ideas and results of this paper have already found applications in Riemannian geometry \cite{HOS} and in Hamiltonian PDE \cite{VicheryRoos}. Section \ref{subsec: app} describes applications to various aspects of symplectic topology, relying on the fact that the spectral norm is a key invariant controlling many quantitative properties of Lagrangian submanifolds. Broadly speaking, the results of this paper connect the a priori disparate fields of quantitative symplectic topology and algebraic topology of loop spaces in a new non-trivial way, making a contribution to both.
%in addition to those to $C^0$ symplectic topology, symplectic homogenization, and the structure of the group of Hamiltonian diffeomorphisms, described below.

 %We shall consider all homological invariants, including the spectral norm, with coefficients in a field $\bK.$ 

The main tool of the paper is the Viterbo isomorphism \cite{Viterbo-iso,SalamonWeber,AbbSchwarz,AbbSchwarz-corr,AbouzaidBook} of BV-algebras between the symplectic cohomology of the cotangent bundle and the homology of the loop space of the base. It is combined with a quantitative study of a generalization of the operations on symplectic cohomology and Lagrangian Floer homology introduced and studied by Seidel and Solomon \cite{SeidelSolomon-q} in the context of mirror symmetry, whereof we provide new calculations in terms of string topology. 

%For general manifolds $L$ we usually assume $\bK= \bF_2,$ while for orientable manifolds whose second Stiefel-Whitney class vanishes on two-tori we may consider coefficients in an arbitrary field $\bK.$ All choices of manifolds and coefficients below are of this kind. We make these restrictions primarily in order to work with trivial local systems in the Viterbo isomorphism \cite{Viterbo-iso,SalamonWeber,AbbSchwarz,AbbSchwarz-corr,AbouzaidBook} of BV-algebras between the symplectic cohomology of the cotangent bundle and the homology of the loop space of the base. However we expect all methods and definitions to adapt to more general situations, yielding similar results (for example Definition \ref{def: pt invertible} below should be taken to depend on a local system on the loop space, as well as on the coefficient field). 

We start by defining a new class of manifolds to which our results apply. Fix a base field $\bK$ as a coefficient ring for homology and cohomology groups. For a closed connected smooth manifold $L$ of dimension $\dim(L) = n,$ consider the constant-loop inclusion map \[\iota: H_\ast (L) \to H_\ast (\cL L),\] and the evaluation map \[ev: H_\ast(\cL L) \to H_\ast (L),\] between its homology and the homology of the free loop space $\cL L$ of $L.$ Given a homogeneous class $a \in H_\ast (\cL L),$ let \[m_a: H_\ast (\cL L) \to H_{\ast+|a|-n+1} (\cL L)\] be the right Chas-Sullivan \cite{ChasSullivan} string bracket $[-,a]$ with $a.$ We recall that the string bracket is given in terms of the Chas-Sullivan product $\ast,$ and the BV-operator \[\Delta: H_\ast (\cL L) \to H_{\ast+1} (\cL L).\] It is essentially the $\Delta$-differential of the product: for homogeneous elements $a,b \in H_{n-\ast}(\cL L),$ \[[b,a] = (-1)^{|b|}(\Delta(b\ast a) - \Delta(b)\ast a - (-1)^{|b|} b \ast \Delta(a)).\] This bracket, together with the product, forms the structure of a Gerstenhaber algebra on $H_{n-\ast}(\cl LL).$ In terms of these operations, the main operator that we consider in this paper is \[P_a: H_*(L) \to H_{*+|a|-n+1}(L)\] \[P_a = ev \circ m_a \circ \iota.\] 
%\[Q_a: H_*(L) \to H_*(L)\] \[Q_a = ev \circ m_a \circ \iota.\] 

%Let  be the BV-operator

%and \[[-,-]: H_\ast (\cL L) \otimes H_{\ast} (\cL L) \to H_{\ast} (\cL L)\] be the string bracket. Recall that it is given by the BV operator and 

%$a_1,\ldots,a_N, b_1,\ldots,b_{N'} \in H_\ast (\cL L)$ 

%and a composition $P$ of 
%$P_{a_1},\ldots, P_{a_N}, Q_{b_1},\ldots, Q_{b_{N'}}$ in a certain order 

\begin{df}\label{def: pt invertible} We call a closed connected smooth manifold $L$ {\em string point-invertible} over $\bK$ if it is $\bK$-orientable and there exists a collection of classes $a_1,\ldots,a_N \in H_\ast (\cL L)$ such that the composition $P = P_{a_N} \circ \ldots \circ P_{a_1}$ satisfies \[[L] = P ([pt]),\] where $[L] \in H_n(L)$ is the fundamental class and $[pt] \in H_0(L)$ is the class of the point. Reformulated more abstractly, $[L] \in \cP([pt]) = \{ P([pt])\,|\, P \in \cP \},$ where $\cP$ is the subalgebra of $Hom(H_*(L),H_*(L))$ generated by $\{P_a\,|\, a \in H_*(\cL L) \}.$ 
\end{df}

\begin{rmk} We note the following three points regarding Definition \ref{def: pt invertible}.
	
\begin{enumerate}[label = \roman*.]  
			
\item Set $H_*(\cL)^+ = \ker(ev: H_*(\cL L) \to H_*(L)).$ It is easy to see that we may, without loss of generality, restrict $a_1,\ldots,a_N$ in Definition \ref{def: pt invertible} to lie in $H_*(\cL)^+$ and replace $\cP$ with its subalgebra $\cP^+$ generated by $\{P_a\,|\, a \in H_*(\cL L)^+ \}.$ Indeed, it is enough to consider homogeneous elements $a,$ in which case $P_a: H_*(L) \to H_{*+|a|-n+1}(L)$ is a homogeneous operator. Further, for all $b \in \iota(H_*(L)),$ $P_b = 0$ since $\Delta = 0$ on $\iota(H_*(L)),$  and $\iota, ev$ are maps of algebras, $H_*(L)$ being endowed with the intersection product. Hence we may correct each homogeneous element $a \in H_*(\cL L)$ by $a_0 = \iota \circ ev(a)$ to obtain the homogeneous element $a' = a-a_0 \in H_*(\cl L)^+,$ with the property that $P_{a'} = P_{a}.$ Hence $[L] = P_{a_N} \circ \ldots \circ P_{a_1} ([pt]),$  $a_1,\ldots,a_N \in H_*(\cl L)$ if and only if $[L] = P_{a'_N} \circ \ldots \circ P_{a'_1} ([pt]),$ with $a'_1,\ldots,a'_N \in H_*(\cl L)^+.$

%with $N$ minimal,

\item For a class $a \in H_*(\cL L)$ we may consider the operator $Q_a: H_*(L) \to H_*(L),$ given by $Q_a = ev \circ m'_a \circ \iota,$ where $m'_a$ is the Chas-Sullivan product by $a.$ The technical arguments in this paper apply to this simpler map, however since $ev$ and $\iota$ are maps of algebras, and $ev \circ \iota = \id,$ we observe that for $x \in H_*(L),$ $Q_a = ev \circ m_a \circ \iota(x) = ev(a \ast \iota(x)) = ev(a) \ast ev\circ \iota( x) = ev(a) \ast x.$ Therefore $Q_a$ is the multiplication operator by $ev(a) \in H_*(L)$ with respect to the intersection product on $H_*(L).$ In particular it does not increase degree. Therefore, while adding the operations $Q_b,$ $b \in H_*(\cL L),$ to Definition \ref{def: pt invertible} may theoretically be useful, in practice it seems to have little effect. 

\item Note that if $P_a: H_*(L) \to H_{*+|a|-n+1}(L)$ increases degree, then the homological degree of $a \in H_*(\cl LL)$ satisfies $|a| \geq n.$ 

%Furthermore for $a,b \in H_*(\cL),$ $x \in H_*(L),$ \[P_{a}\circ Q_{b} (x) = ev  \circ m_a \circ \iota (ev(b) \ast x) = ev \circ \Delta ((a \ast \iota(ev(b)) \ast \iota(x)).\] Hence $P_a \circ Q_b = P_{a \ast \iota(ev(b))},$ and therefore adding the operations $Q_b,$ $b \in H_*(\cL L),$ to Definition \ref{def: pt invertible} has no effect. 

	\end{enumerate}
\end{rmk}

%\red{Can one do better with Gerstenhaber?}
By a result of Menichi \cite{Menichi-spheres} the class of string point-invertible manifolds contains spheres of odd dimension $S^{2m+1},$ $m \geq 0,$ with arbitrary coefficients, and $S^{2}$ with coefficients in $\bF_2.$ A minor modification of the argument of Menichi for $S^2$ shows that the even-dimensional spheres $S^{2m}, m \geq 1$ are in this class, with $\bF_2$ coefficients. Furthermore, we have the following general structural result for this class.

\begin{prop} \label{prop: prod}
	The class of string point-invertible manifolds over a fixed field $\bK$ is closed under products.
\end{prop}

%Firstly, by a result of Westerland \cite{Westerland-string-sphere-proj}, this class contains $\{\R P^n, S^n, \C P^n, \HH P^n\}$ for all $n \geq 1,$ where for $\R P^n$ we should restrict coefficients to $\bF_2,$ by Tamanoi \cite{Tamanoj-string} it contains the complex Stiefel manifolds $V_{n+1-k}(\C^{n+1}) \cong SU(n+1)/SU(k)$ of orthonormal $(n+1-k)$-frames in $\C^{n+1}$ for all $n \geq 0,$ $0\leq k \leq n,$ by Hepworth \cite{Hepworth-string}, it contains the compact connected Lie groups, and by Berglund and B\"{o}rjeson \cite{BorgBer-string}, in characteristic zero it contains $(n-1)$-connected manifolds of dimension at most $3n-2,$ $n \geq 2.$ Furthermore, by a result of Vaintrob \cite{MitkaV-string}, this class does {\em not} contain the closed surface $\Sigma_g$ of genus $g,$ for each $g > 1.$ Finally, in Section \ref{subsec: proof prod} we show the following abstract result.

In particular, the $n$-torus $T^n$ is string point-invertible over any field $\bK.$ To verify the definition one can take (the image under coefficient change to $\bK$ of) the sequence $a_1,\ldots,a_n$ of positive generators of $H_n(\cL_{e_j} L;\Z) \cong \Z$ for free homotopy classes of loops $e_1,\ldots,e_n$ corresponding to a positively oriented basis of $\Z^n.$ The main result of this paper is the following.

\begin{thm}\label{thm:Vit-prod}
Let $L$ be string point-invertible over a field $\bK.$ Let $g$ be a Riemannian metric on $L.$ Then there exists a constant $C(g,L;\bK)$ such that for all exact Lagrangian submanifolds $L_0, L_1$ containted in the unit codisk bundle $D^*_g L \subset T^*L,$ the spectral norm of the pair $L_0,L_1$ satisfies \[\gamma(L_0,L_1;\bK) \leq C(g,L;\bK).\]  
\end{thm}

\begin{rmk}
By the triangle inequality for the spectral norm, it is enough to prove the above statement for $L_1 = L,$ the zero section in $D^*_g L.$
\end{rmk}

%\red{discuss sharpness - easy to see that the bound for $T^n$ is sharp.}

This statement was previously known for $\bK = \bF_2,$ and $L \in \{\R P^n, \C P^n, \HH P^n, S^n \,:\, n\geq 1\}$ by \cite{S-Zoll}, essentially in the case when $L_1$ is Hamiltonianly isotopic to the zero section, and $L_0 = L.$ In particular the case of $T^n$ for $n>1$ has remained completely open. We note that by the examples in Section \ref{subsec: examples}, Theorem \ref{thm:Vit-prod} is rather complementary to the result of \cite{S-Zoll}. Furthermore, while quite a few manifolds including surfaces of higher genus are not string point-invertible, and hence Theorem \ref{thm:Vit-prod} does not apply as such, there is an ongoing work \cite{BC-private} proving related results for bases given by arbitrary closed connected manifolds. 

%Finally, compare: $\gamma(df,L) = \max f - \min f \leq \left|\left|df\right|\right| \cdot \mrm{diam}(L,g)$ for $f \in \sm{L,\R}.$ 

The strategy of the proof of Theorem \ref{thm:Vit-prod} differs significantly from that of \cite{S-Zoll}, in particular in that it does not make use of compactification arguments, relying instead on a kind of "wrong way Lusternik-Schnirelmann inequalities". The main idea is threefold: first, a homogeneous class $a \in H_{n-\ast}(\cl LL)^+$ corresponds by the Viterbo isomorphism \cite{Viterbo-iso,SalamonWeber,AbbSchwarz,AbbSchwarz-corr,AbouzaidBook} to a class $\alpha \in  SH^*(\cl LL)$ (the latter computed with suitable background class) with the property that $r_{L'}(\alpha) = 0 \in HF^*(L',L')$ for each exact Lagrangian $L' \subset T^*L,$ where $r_{L'}: SH^*(\cl LL) \to HF^*(L',L')$ is the natural closed-open restriction map (by \cite{FukayaSeidelSmith-cotangentsc,Kragh-nearby,Abouzaid-nearbyMaslov,AK-simplehomotopy} each such $L'$ has vanishing Maslov class, is Spin relatively to the above background class, and endowed with suitable Spin structure is Floer-theoretically equivalent to the zero section $L$). Second, working up to $\epsilon > 0,$ given that $L_0,L_1 \subset D = D^*_g L,$ a Liouville domain with contact boundary $S = \partial D = S^*_g L,$ the work of Seidel and Solomon \cite{SeidelSolomon-q} gives an operation $HF^*(L_0,L_1) \to HF^{*+|\alpha|-1}(L_0,L_1),$ which raises the action filtration by no more than a symplectic-homological spectral invariant $c(\alpha,D,S)$ corresponding to the class $\alpha$ and the domain $D.$ Finally, using further TQFT operations for Lagrangian Floer cohomology and symplectic cohomology \cite{AbouzaidBook, AbbSchwarz}, we calculate that under the Floer-theoretic equivalence with the zero section, Poincar\'{e} duality $H^*(L) \cong H_{n-\ast}(L),$ and the Viterbo isomorphism, this operation is given by $P_a:H_{\ast}(L) \to H_{\ast +|a|-n +1}(L).$ Therefore, in view of string point-invertibility, assuming for simplicity that all $P_{a_j},$ $1 \leq j \leq N,$ increase degree, which tends to happen in practice, by successively writing inequalities that bound the Lagrangian spectral invariants of classes of higher homological degree in terms of those of classes of lower homological degree, we arrive to a uniform upper bound on the spectral distance $\gamma(L_0,L_1; \bK),$ finishing the proof.

\begin{rmk}\label{prop: bound} A few remarks on Theorem \ref{thm:Vit-prod} are in order. \begin{enumerate}[label = \roman*.]
\item It is not necessary that our Weinstein domain be $D^*_g L$ for a Riemannian metric $g.$ In fact the same result holds for any Weinstein domain $D$ containing $L$ with completion given by $T^*L.$ In this case the upper bound will be given in terms of a constant $c(D,L;\bK).$ For example $D$ may be given by a Finsler metric, or an optical domain: one that is strictly fiberwise star-shaped, and has a smooth boundary. Finally, approximating general, not necessarily smooth, strictly fiberwise starshaped domains by ones with smooth boundary, we obtain a uniform bound in that case as well.
\item In fact $C(g,L;\bK)$ in Theorem \ref{thm:Vit-prod} can be chosen to be equal to a certain sum of spectral invariants relative to the domain $D$ with boundary $S,$ corresponding to any $N$-tuple $a_N,\ldots,a_1 \in H_*(\cl LL)$ as in Defintion \ref{def: pt invertible}. See Equation \eqref{eq: bound on spectral norm}. Furthermore, it is easy to see that the spectral invariants $c(a,D,S)$ are continuous in the Banach-Mazur distance with respect to the natural $\R_{>0}$-action on $T^*L$ \cite{VukasinJun,PolterovichRSZ-book} (see \eqref{eq: Jun-Vukasin Lipschitz}), and hence extend for example to the non-smooth strictly fiberwise star-shaped case. As a consequence we obtain bounds in the non-smooth fiberwise star-shaped case in terms of the extension of the spectral invariants.
\item Let $g_0$ be the standard metric of diameter $1/2$ on $S^1 = \R/\Z.$ Let $D_0 = D^*_{g_0} S^1 = [-1,1] \times S^1.$ From \eqref{eq: bound on spectral norm}, it is evident that $C(D_0,L_0;\bK)=1$ in this case. This upper bound is sharp, since for each $\epsilon > 0$ sufficiently small, it is easy to construct a Lagrangian $L'_0 \subset D_0$ Hamiltonian isotopic to $L_0=S^1$ in $D_0,$ with $\gamma(L'_0,L_0) > 1 -\epsilon,$ and the intersection $L'_0 \cap L_0$ is transverse and consists of precisely $2$ points $x,y$ of index $1$ and $0$ respectively. Consider now the stricly fiberwise star-shaped domain $D \subset T^*(T^n)$ given by \[D = (D_0)^n = [-1,1]^n \times T^n.\] It is easy to calculate that the upper bound obtained by continuity from \eqref{eq: bound on spectral norm} is in this case $C(D,L;\bK) = n.$ It is seen to be sharp by noting that $L' = (L'_0)^n$ satisfies $\gamma(L',L) = n\cdot \gamma(L'_0,L_0),$ since $L' \pitchfork L,$ and the only intersection point of $L'$ and $L$ indices $n$ and $0$ are $(x,\ldots,x),$  and $(y,\ldots, y)$ respectively.

\item We note that Theorem \ref{thm:Vit-prod} fails for general bounded Liouville domains. For example it is false for Lagrangians Hamiltonian isotopic to $L$ in plumbings of $D^*L$ with two or more cotangent disk bundles by \cite{ZapPlumbing}.
	\end{enumerate}

\end{rmk}

\subsection{String point-invertibility: examples and non-examples}\label{subsec: examples}

We discuss the size of the class of string point-invertible manifolds by describing examples and non-examples, based on known calculations of the Chas-Sullivan Gerstenhaber algebra, in addition to the results of Menichi \cite{Menichi-spheres} and Proposition \ref{prop: prod} mentioned above. These calculations turn out to be quite delicate, and to depend on the choice of coefficients, and hence so does the property of string point-invertibility.

By a result of Tamanoi \cite{Tamanoi-Stiefel}, the complex Stiefel manifolds $V_{n+1-k}(\C^{n+1}) \cong SU(n+1)/SU(k)$ of orthonormal $(n+1-k)$-frames in $\C^{n+1}$ for all $n \geq 0,$ $0\leq k \leq n,$ are string point-invertible over arbitrary coefficients, and by Menichi \cite{Menichi-double} (see related result of Hepworth \cite{Hepworth-Lie}), all compact connected Lie groups are string point-invertible, with characteristic zero coefficients. 

%Results of Westerland \cite{Westerland-spheres-proj} indicate that over $\bF_2,$ certain projective spaces, for example $\C P^n$ and $\HH P^n,$ where $n$ is odd, are also string point-invertible. Even though the results of \cite{Westerland-spheres-proj} pertain to a different BV-operator than the one we use here, since these spaces are simply connected, 

We note that as string point-invertibility depends only on the Gerstenhaber algebra structure on $H_{n-\ast}(\cl LL)$ and the evaluation and inclusion maps, at least for $L$ simply connected, by the results of Malm \cite{Malm} and Felix-Menichi-Thomas \cite{FelixMenichiThomas} or Keller \cite{Keller-derived}, as well as \cite{FelixThomas-BV} (see also the discussion in \cite{LekiliEtgu}), it depends only on the singular cochain dg-algebra $C^*(L)$ of $L,$ up to isomorphism. 

%$L$ is simply connected results of Malm \cite{Malm} and Felix-Menichi-Thomas \cite{FelixMenichiThomas} or Keller \cite{Keller-derived} imply that the Gerstenhaber algebra of interest is determined  (see also the discussion in \cite{LekiliEtgu}).
%by a result of F\'{e}lix and Thomas .  
%
% however . While this is not expected to influence the Gerstenhaber algebra structure, we have not found the suitable modified calculations in the literature. 

%Results of Westerland \cite{Westerland-spheres-proj} seem to indicate that over $\bF_2,$ certain projective spaces, for example $\C P^n$ and $\HH P^n,$ where $n$ is odd, are also string point-invertible, however these results pertain to a different BV-operator than the one we use here. While this is not expected to influence the Gerstenhaber algebra structure, we have not found the suitable modified calculations in the literature. 

By Menichi \cite{Menichi-spheres}, this class does {\em not} contain the even-dimensional spheres $S^{2m}, m\geq 1$ for coefficients of characteristic zero, for instance, and the same is true for $\C P^n, n\geq 1,$  $\HH P^n, n \geq 1$ and $\mathbb{O}P^2$ by results of Yang \cite{Yang-BV}, Chataur-Le Borgne \cite{ChataurLeBorgne-complexproj}, Hepworth \cite{Hepworth-complexproj}, and Cadek-Moravec \cite{CadekMoravec-quaternionicproj}. Moreover, by Westerland \cite{Westerland-spheres-proj} and \cite{ChataurLeBorgne-complexproj,Hepworth-complexproj,CadekMoravec-quaternionicproj} the same is true for $\C P^n, \HH P^n,$ where $n>1$ and for $\mathbb{O} P^2,$ with $\bF_2$-coefficients. This question is further explored in \cite{HOS} with arbitrary coefficients. By a result of Vaintrob \cite{MitkaV-string}, this class does not contain the closed surface $\Sigma_g$ of genus $g,$ for each $g > 1,$ and any choice of coefficients. The same is true for closed manifolds of strictly negative sectional curvature, again by \cite{MitkaV-string} or by an index argument of Tonkonog \cite{DT-private}.

\subsection{Applications}\label{subsec: app}

\subsubsection{$C^0$ continuity of the Hamiltonian spectral norm}

As observed in \cite{S-Zoll}, an argument of neck-stretching around a divisor in $M \times M^-$ that makes the Lagrangian diagonal $\Delta_{M} \subset M \times M^-$ exact, where $(M,\om)$ is a closed symplectic manifold such that Theorem \ref{thm:Vit-prod} holds for $L=M,$ and $M^-$ denotes the symplectic manifold $(M,-\om),$ allows one to prove, for example in the symplectically aspherical case, that the Hamiltonian spectral norm on $\Ham(M,\om)$ is Lipschitz in the $C^0$ norm, in a $C^0$-neighborhood of the identity. We pick one instance of such an application. The $C^0$-distance between two diffeomorphisms $\phi_0,\phi_1$ of $M$ is defined as $d_{C^0}(\phi_0,\phi_1) = \max_{x \in M} d(\phi_0(x),\phi_1(x)),$ the distance $d$ being taken with respect to a background Riemannian metric on $M.$

%\red{Change $\gamma_{pt}$ to $\gamma$ and switch weakly aspherical case to weakly monotone: use the homological perturbation argument}

\begin{cor} \label{cor: C^0 continuity}
Let $g$ be a Riemannian metric on $T^{2n},$ and $\bK$ be a field. The spectral norm \[\gamma: \Ham(T^{2n},\om_{st}) \to \R_{\geq 0}\] over $\bK$ satisfies the following. There exist constants $C,\delta > 0,$ such that \[\min \{\gamma(\phi), \delta C\} \leq C \cdot d_{C^0}(\phi,1)\] for all $\phi \in \Ham(T^{2n},\om_{st}).$ 
\end{cor}

We refer to \cite{S-Zoll} for a discussion of results of this kind, such as \cite{Seyfaddini-descent, BHS-spectrum, S-Zoll} and their applications \cite{LSV-conj, KS-bounds, BHS-spectrum}, and a proof of a similar implication \cite[Theorem C]{S-Zoll}. We observe that further such results in the setting of closed monotone symplectic manifolds that are string point-invertible as smooth manifolds are not difficult to deduce. However, for reasons of conciseness, we defer this discussion to a further publication.

\subsubsection{Quasi-morphisms on the Hamiltonian group of cotangent disk bundles}

Similarly to \cite{S-Zoll}, Theorem \ref{thm:Vit-prod} yields the existence of non-trivial homogeneous quasi-morphisms on $\Ham_c(D^*_g L)$ for $L$ string point-invertible, providing new examples of quasi-morphisms and quasi-states on compactly supported Hamiltonian diffeomorphism groups of Weinstein domains. The notion of such quasi-morphisms, which are maps to $\R$ that are additive up to a uniformly bounded error, played an important role in symplectic topology (see \cite{EntovPolterovichCalabiQM,Entov-ICM,Lanzat-qm,BEP-ball,S-Zoll} for example). This in turn has applications to the geometry of the Poisson bracket of compactly supported functions in $D^*_g L.$ We remind the reader that a quasi-morphism is called homogeneous if it is additive on all abelian subgroups, and non-trivial if it is not a homomorphism. We summarize this application as follows, and refer to \cite{S-Zoll} for its deduction from from Theorem \ref{thm:Vit-prod}.

\begin{cor}\label{cor:qm intr}
Let $L$ be string point-invertible over $\bK.$ The map $\mu: \til{\Ham}_c(D^* L) \to \R$ on the universal cover of $\Ham_c(D^*L)$ given by \[\mu([\{\phi^t_H\}]) = \lim_{k \to \infty} \frac{1}{k} c([L], H^{(k)}),\] for  $H^{(k)}(t,x) = k H(kt,x),$ $k \in \Z_{>0}$ descends to a well-defined non-zero homogeneous quasimorphism  $\mu: \Ham_c(D^* L) \to \R.$ Moreover $\mu$ vanishes on each element $\phi \in \Ham_c(D^* L)$ such that $\supp(\phi)$ is displaceable by an element of  $\Ham_c(D^*L).$ For $F,G \in C^{\infty}_c (D^*L, \R),$ the map $\zeta: C^{\infty}_c (D^*L, \R) \to \R$ by $\zeta(H) = \mu(\phi^1_H),$ satisfies \begin{equation}\label{eq:zeta pb}|\zeta(F+G) - \zeta(F) - \zeta(G)| \leq \sqrt{2 C(g,L) ||\{F,G\}||_{C^0}},\end{equation} where $\{F,G\}$ is the Poisson bracket of $F,G.$ In particular, whenever $\{F,G\} = 0,$ we obtain \[\zeta(F+G) = \zeta(F) + \zeta(G).\]
\end{cor}

These maps were defined, and shown to enjoy various interesting properties in \cite[Theorems  1.3 and 1.8, Propositions 1.4 and 1.9]{MonznerVicheryZapolsky}, yet the quasi-morphism property, while anticipated therein, was hitherto known only for $L \in \{ \R P^n, \C P^n, \HH P^n, S^n \}$ covered in \cite{S-Zoll}. Moreover, in the new case of $T^n,$ the quasimorphism $\mu: \Ham_c(D^*_g T^n) \to \R$ in the case of $T^n$ is immediately seen to be invariant under finite coverings $T^n \to T^n,$ scaled suitably, as defined in \cite{Viterbo-homog} (see also \cite{MonznerVicheryZapolsky}). A similar invariance holds for products $T^n \times L$ with $L$ string point-invertible, with the induced coverings, or for finite coverings $L' \to L$ with both $L,L'$ string point-invertible. It is an interesting topological question to determine whether or not the class of string point-invertible manifolds is closed with respect to finite coverings: we expect this to be the case when working with coefficients of characteristic zero. Furthermore, Theorem \ref{thm:Vit-prod} provides a different proof, and indeed a strengthening, of the results of \cite[Section 7]{Viterbo-homog}.

%\red{figure out the neck-stretching argument for general $L \subset M$ weakly homologically monotone}

\subsubsection{Hausdorff-continuity of the Lagrangian spectral norm}

We finish with yet another application, that is proved again by a neck-stretching argument (see \cite[Theorem F]{S-Zoll}), combined with arguments related to non-trivial fundamental groups. %\red{update to monotone case; using homological perturbation}

\begin{cor}\label{cor: embedded bound}
Let $L$ be string point-invertible over $\bK.$ Suppose $L$ is embedded as a weakly exact $\pi_1$-injective Lagrangian submanifold in a symplectically aspherical symplectic manifold $M$ that is closed or tame at infinity. Consider the pair $(U,L),$ for $U \subset M$ a Weinstein neighborhood of $L,$ that is symplectomorphic to the pair $(D,0_{L}),$ for a Weinstein domain $D \subset T^*L$ containing the zero-section $0_L \subset T^*L.$ Consider $r \in (0,1),$ and let $U^r$ be the preimage of $r \cdot D$ by the symplectomorphism. Then there exists a constant $C(D,L;\bK)$ such that if $L' \subset M$ is a Hamiltonian image of $L$ that is contained in $U^r,$ then \[\gamma(L',L;\bK) \leq  C(D,L;\bK) \cdot r.\]
\end{cor}

An example of the situation described in Corollary \ref{cor: embedded bound} is the torus $L = T^n$ embedded as $L_1 \times \ldots \times L_n$ inside $\Sigma_1 \times \ldots \times \Sigma_n,$ where for all $1 \leq j \leq n,$ the submanifold $L_j \subset \Sigma_j$ is an embedded simple closed curve in the closed oriented surface $\Sigma_j$ of genus at least $1,$ that does not bound a disk. The condition on the embedding holds when $\pi_2(M,L) = 0.$ In the cases of Corollaries \ref{cor: C^0 continuity} and \ref{cor: embedded bound}, arguments following \cite[Theorem B]{KS-bounds} show that the associated Floer-theoretic barcodes, a notion that has recently attracted much attention in symplectic topology (see \cite{PolShe} and for example \cite{UsherZhang, Zhang, PolSheSto, PolterovichRSZ-book, Team,stevenson, KS-bounds, LSV-conj, BHS-spectrum, VukasinJun, Usher-BM, RizzSull-pers, S-HZ}), up to shift, are continuous in the Hausdorff metric on the Lagrangians considered as subsets of $M$ with respect to a background Riemannian metric (see \cite{S-Zoll} for a discussion of this kind of result, which in particular does not follow from $C^0$ continuity of the Hamiltonian spectral norm). Moreover, one can deduce analogues of Corollary \ref{cor: embedded bound} for certain monotone Lagrangian submanifolds, however for reasons of conciseness we defer this discussion to a further publication.

\subsubsection{Outlook}

As a closing remark, we mention that it would be very interesting to see if additional algebraic structures on symplectic cohomology and string topology could be applied to extend the class of manifolds $L$ for which Viterbo's conjecture holds. For instance, introducing suitable local systems on $\cl LL,$ or considering higher operations in suitable $L_{\infty}$-algebras or SFT algebras, may yield further such examples. 

%that are trivial when restricted to the image of the constant-loop embedding $L \to \cl LL,$
%(see \cite{TonkonogDescendants} and references therein, for example)

\section{Preliminaries}\label{sec: prelim}

Throughout the paper we follow the definitions and notations of Seidel and Solomon \cite{SeidelSolomon-q}, with one distinction: we take the opposite sign for all action functionals. Furthermore, we adopt the following convention: everywhere we argue up to $\epsilon,$ and allow arbitrarily small perturbations of all Hamiltonian terms involved. For example, when the Hamiltonian perturbation data has curvature zero, it means that we may achieve regularity by a Hamiltonian term arbitrarily close to the given one, in such a way as to make the curvature arbitrarily small.

We sketch the part of definitions where additional detail is required. In particular we look at exact Lagrangian submanifolds $L$ inside a Weinstein manifold $W$ with Liouville form $\theta,$ and symplectic form $\omega = d\theta.$ We restrict attention to the case when $W$ is the completion of a Weinstein domain $D$ with compact contact boundary $S,$ and we consider $L \subset D.$ 

For the definition of symplectic cohomology we choose a cofinal family of Hamiltonians $H_{\lambda}$ that are $\epsilon$-small in the $C^2$ norm on $D,$ and are in fact non-positive Morse functions there with gradient pointing outward of $D$ at $S.$ Furthermore outside of $D \cup C$ for a small collar neighborhood $C = C_{\lambda}$ of $S,$ $H_{\lambda} = \lambda \cdot r,$ where $r$ is the radial coordinate on the infinite end $([1,\infty) \times S, d(r\alpha)),$ $\alpha = \theta|_S,$ of the completion, with the property that $\lambda \notin \mrm{Spec}(\alpha,S),$ that is, it is not a period of a closed Reeb orbit of $\alpha.$ The latter is a smooth loop $\gamma:\R/T\Z \to S,$ $T > 0,$ such that $\gamma'(t) = R_{\alpha} \circ \gamma(t)$ for all $t \in \R/T\Z,$ and the Reeb vector field $R_{\alpha}$ on $S$ is defined by the conditions $\iota_{R_{\alpha}} \alpha = 1,$ $L_{R_{\alpha}} \alpha = \iota_{R_{\alpha}} d\alpha = 0.$ Furthermore, we require that $0 < \epsilon \ll \epsilon_{\alpha} = \min \mrm{Spec}(\alpha,S),$ and that $H_{\lambda}$ be radial increasing and convex in $C.$  Furthermore, (an arbitrarily small perturbation in $D \cup C$ of) $H_{\lambda}$ is non-degenerate at all its $1$-periodic orbits, which necessarily lie in $D \cup C.$ All closed $H_{\lambda}$ one-periodic orbits in $C$ are in a $2$ to $1$ correspondence with the Reeb orbits of $\alpha$ of periods in $[\epsilon_{\alpha},\lambda),$ and we choose $C, H_{\lambda}$ so that for a fixed $\delta > 0$ independent of $\lambda$ the $H_{\lambda}$-actions of these orbits are $\delta$-close to their $\alpha$-periods. We choose $\delta \ll \epsilon_{\alpha}.$ Furthermore, we require that $H_{\lambda_k} \leq H_{\lambda_{k+1}},$ $k \geq 1,$ on $W$ for a strictly increasing sequence $\{\lambda_k\}_{k \geq 1},$ $\lambda_k \xrightarrow{k \to \infty} \infty,$ in $\R_{>0} \setminus \mrm{Spec}(\alpha,S),$ and that $||H_{\lambda_k}||_{C^2(D)} \xrightarrow{k \to \infty} 0.$ That these choices can be made is standard material on symplectic cohomology (see for example \cite[Section 5]{HutchingsGutt-cube}). From now on, when we write $H_{\lambda}$ we assume that $\lambda = \lambda_k$ for some $k \geq 1.$ 

Finally, for two fixed Lagrangian submanifolds $L_0,L_1 \subset D$ we may choose $H_{\lambda_k}$ on $D$ so that the intersection $\phi^1_{H_{\lambda_k}} (L_0) \cap L_1$ is transverse for all $k \geq 1.$ We recall that the $\omega$-compatible almost complex structures $J$ that we consider are of convex type: on the infinite end of $W,$ $J \del_r = R_{\alpha},$ and $J$ is invariant under translations in $\rho = \log(r).$ The action of a periodic orbit $x$ of $H$ is defined as \[\cA_H(x) = -\int_0^1 H(t,x(t))\,dt + \int_x \theta.\] %We recall that we consider the Hamiltonian vector field $X^t_H$ of $H$ defined by $\iota_{X_H^t} \om = -d (H(t,-)).$

We consider the Floer cohomology groups $CF^*(H_{\lambda}),$ that as $\bK$-modules have generators corresponding to $1$-periodic orbits of $H_{\lambda},$ the coefficient near $x_-$ of whose differential $d_{H;J}$ evaluated on $x_+,$ for $J$-generic, counts isolated solutions $u: \R \times S^1 \to W$ to the Floer equation \[ \del_s u + J_t(u) (\del_t u - X_H^t(u)) = 0,\] with asymptotic conditions $u(s,-) \to x_{\pm}(-),$ as $s \to \pm\infty,$ for $1$-periodic orbits $x_{\pm}$ of $H = H_{\lambda}.$ Here $X_H$ is the time-dependent Hamiltonian vector field of $H$ given by $\iota_{X_H^t} \omega = - d(H(t,-)).$ Note that the critical points of $\cl A_H$ on the loop space $\cL W$ are precisely given by time-1 periodic orbits of the isotopy $\{\phi^t_H\}$ generated by $X_H.$ Furthermore, if $d_{H;J}(y) = z,$ then $\cl A_H(y) > \cl A_H(z).$ Finally, $CF^*(H_{\lambda_k})$ forms a direct system with respect to the natural order on $\{\lambda_k\},$ by means of Floer continuation maps: $CF^*(H_{\lambda_k}) \to CF^*(H_{\lambda_{k'}})$ for $\lambda_k \leq \lambda_{k'}.$ Here it is important that $H_{\lambda_k}(t,x)$ is increasing as a function of $k.$ The symplectic cohomology of $W$ is defined as \[SH^*(W) = \lim_{\rightarrow} CH^*(H_{\lambda}) = \lim_{\rightarrow} CH^*(H_{\lambda_k}).\] Its filtered version associated to $(D,S)$ is defined as \[SH^*(W)^{< t} = \lim_{\rightarrow} CH^*(H_{\lambda_k})^{< t},\] where $CH^*(H_{\lambda_k})^{< t}$ is the subcomplex generated by 1-periodic orbits of action strictly smaller than $t.$

Given two exact Lagrangian submanifolds $L_0, L_1 \subset D,$ we choose generic perturbation data $\cD = (J^{L_0,L_1}, K^{L_0,L_1})$ consisting of an almost complex structure $J^{L_0,L_1}_t$ that depends on time $t \in [0,1],$ and a Hamiltonian $K^{L_0, L_1}$ that is radial outside of $D \cup C$ (for example zero there), and define the Floer complex $CF(L_0,L_1; \cD)$ with generators corresponding to $X_{K^{L_0,L_1}}$-chords from $L_0$ to $L_1$, the matrix coefficients $\left< d_{L_0,L_1; \cD} (x_+) , x_-\right>$ of whose differential $d_{L_0,L_1; \cD}$ count isolated solutions $u: \R \times [0,1] \to W$ to the Floer equation \[\del_s u + J^{L_0,L_1}_t(u)(\del_t u - X^t_H(u)) = 0,\] with boundary conditions \[u(\R,0) \subset L_0, \; u(\R,1) \subset L_1,\] and uniform asymptotics \[ u(s,-) \xrightarrow{s \to \pm \infty} x_{\pm}(-).\]
%\red{or the other direction?}

Enhancing $L_0, L_1$ to $\ul L_0 = (L_0, f_0), \ul L_1 = (L_1, f_1)$ by choices of primitives $f_0 \in \sm{L_0, \R},$ $f_1 \in \sm{L_1, \R},$ we define the action functional on the space of paths $\cl P(L_0, L_1)$ in $W$ from $L_0$ to $L_1,$ \[\cA_{\ul L_0, \ul L_1;\cl D}: \cl P(L_0, L_1) \to \R\] \[\cA_{\ul L_0, \ul L_1;\cl D}(x) = - \int_{0}^{1} K^{L_0,L_1}(t,x(t)) + \int_x \theta + f_1(x(1)) - f_0(x(0)).\]

The critical points of $\cA_{\ul L_0, \ul L_1;\cl D}$ correspond to the generators of $CF(L_0,L_1;\cD),$ and if $d_{\ul L_0, \ul L_1; \cl D}(y) = z$ then $\cA_{\ul L_0, \ul L_1;\cl D}(y) > \cA_{\ul L_0, \ul L_1;\cl D}(z).$ Furthermore, as we assume that $L_0, L_1$ are connected, $\cA_{\ul L_0, \ul L_1;\cl D}$ does not depend on the enhancements $\ul L_0, \ul L_1$ of $L_0, L_1$ up to an additive constant.

%Now we consider $CF^*(H_{\lambda}), 

%$CF^*(L,L), CF^*(L_0,L_1).$ 

%\red{sketch definition of Floer homology, Lagrangian Floer homology, %monotone continuation maps, symplectic homology}\\

%\red{define all actions carefully!}\\

For a class $a \in SH^*(W)\setminus \{0\},$ its symplectic cohomology spectral invariant $c(a,D,S)$ relative to the domain $D$ with contact-type boundary $S,$ is defined as \[c(a,D,S) = \inf\{t \in \R\,|\, a \in \ima\left( SH^*(W)^{<t} \to SH^*(W) \right) \},\] where $SH^*(W)^{<t} \to SH^*(W)$ is the natural map induced by the 
%\[c(a,D,S) = \inf\{t \in \R\,|\, \pi^{\infty,t}(a) = 0 \in SH^*(W)^{\geq t} \},\] where $\pi^{\infty,t}: SH^*(W) \to SH^*(W)^{\geq t}$  is the natural map induced by the projections $CF^*(H_{\lambda}) \to CF^*(H_{\lambda})^{\geq t},$ exhibiting $CF^*(H_{\lambda})^{\geq t}$ as a quotient complex of 
inclusions of complexes $CF^*(H_{\lambda})^{< t} \subset CF^*(H_{\lambda}).$ These spectral invariants (see \cite{VukasinJun,JunyoungLee-Hill,PolterovichRSZ-book}) are known to satisfy the following properties. First, $c(a,D,S)$ is given as $\int_{\gamma}\alpha_S$ for a certain $\alpha_S$-Reeb orbit $\gamma$ on $S.$ In particular $c(a,D,S) > 0.$ Second, $c(a,D,S)$ is monotone with inclusions of Liouville domains $D \subset D',$ with completion $W$ (\cite{VukasinJun},\cite[Section 8]{HutchingsGutt-cube}). Finally, for $t \in \R,$ \[c(a,\psi^t D, \psi^t S) = e^t c(a, D, S)\] where $\psi^t$ is the flow of the Liouville vector field $X$ given by $\iota_X \omega = \lambda.$ In particular if $\psi^{-t} D \subset D' \subset \psi^t D$ then \begin{equation}\label{eq: Jun-Vukasin Lipschitz}|\log c(a,D,S) - \log c(a,D',S')| \leq t.\end{equation}

For a class $x \in HF^*(L_0,L_1) \setminus \{0\}$ its spectral invariant $c(x,\ul L_0, \ul L_1; \cD)$ relative to the enhancements $\ul L_0, \ul L_1$ and perturbation data $\cD,$ is set to be \[ c(x,\ul L_0, \ul L_1; \cD) = \inf\{t \in \R\,|\, x \in \ima\left(HF^*(L_0, L_1; \cD)^{<t} \to HF(L_0, L_1; \cD)\right) \},\]  %\[ c(x,\ul L_0, \ul L_1; \cD) = \inf\{t \in \R\,|\, \pi^{\infty,t}(x) = 0 \in HF^*(L_0, L_1\; \cD)^{\geq t} \},\] where $\pi^{\infty,t}: HF^*(L_0,L_1; \cD) \to HF^*(L_0,L_1)^{\geq t}$ is induced by the quotient projection \[CF^*(L_0,L_1; \cD) \to CF^*(L_0,L_1)^{\geq t} = CF^*(L_0,L_1; \cD)/ CF^*(L_0,L_1; \cD)^{<t},\] 
where $HF(L_0, L_1; \cD)^{<t}$ is the homology of the subcomplex $CF^*(L_0,L_1; \cD)^{<t}$ of $CF^*(L_0,L_1; \cD)$ generated by chords $z$ of action $\cl A_{\ul L_0, \ul L_1}(z) < t.$ It is well-known (see \cite{LeclercqZapolsky,S-Zoll} and references therein) that $c(x,\ul L_0, \ul L_1; \cD)$ is given by $\cl A_{\ul L_0, \ul L_1; \cD}(z)$ for a generator $z$ of $CF^*(L_0,L_1; \cD),$ and is therefore finite. Furthermore, $c(x,\ul L_0, \ul L_1; \cD)$ does not depend on the almost complex structure part $J^{L_0,L_1}$ of $\cD,$ and is Lipschitz in the Hofer norm of the Hamiltonian term $K^{L_0,L_1}$ of $\cD,$ in the sense that if the Hamiltonian terms $K,K'$ of $\cD, \cD'$ agree outside a compact set (in our case this means that their slopes at infinity agree), then \[|c(x,\ul L_0, \ul L_1; \cD) - c(x,\ul L_0, \ul L_1; \cD')| \leq \int_0^1 (\max_W (F_t) - \min_W (F_t))\,dt,\] where $F = K' \# \overline K$ is the Hamiltonian generating the flow $\phi^t_{K'} \circ (\phi^t_K)^{-1}.$ This allows us to extend the spectral invariant to arbitrary perturbations (even continuous ones), and in particular we define $c(x, \ul L_0, \ul L_1)$ as the limit of $c(x,\ul L_0, \ul L_1; \cD)$ as the norm of the Hamiltonian term of $\cD$ tends to zero. Finally, we remark that if $\phi^t_K(L_0) \subset D \setminus C,$ for all $t \in [0,1],$ and where $C$ is the collar neighborhood of $S$ such that $K$ is convex radial in $C$ and has slope $\lambda$ outside $D \cup C,$ then $c(x, \ul L_0, \ul L_1; \cD)$ depends only on $K(t,x)$ for $(t,x) \in [0,1] \times (D \setminus C),$ by a suitable maximum principle. Indeed, in this case the filtered Floer complex $(CF^*(L_0,L_1; \cD), \cl A_{\ul L_0, \ul L_1; \cD})$ does not depend on $K(t,x)$ for $(t,x) \notin [0,1] \times (D \setminus C).$ In particular, if the Hamiltonian term of $\cD_k$ is given by $H_{\lambda_k},$ then \begin{equation}\label{eq: turning off slope} c(x,\ul L_0, \ul L_1; \cD_k) \xrightarrow{k \to \infty} c(x,\ul L_0, \ul L_1).\end{equation}  From now on, for each exact Lagrangian $L$ we fix an enhancement $\ul L,$ and set $c(x,L_0,L_1; \cD) := c(x, \ul L_0, \ul L_1; \cD),$ and $c(x,L_0,L_1) := c(x, \ul L_0, \ul L_1).$ Our results will not depend on this choice.

\bs

Signs in the count of the differentials, as well as gradings, in both kinds of Floer complexes are determined by certain background classes. We summarize these below, and refer to \cite{SeidelSolomon-q},\cite{AK-simplehomotopy},\cite{Abouzaid-nearbyMaslov},\cite{Kragh-nearby},\cite{SeidelBook},\cite{AbouzaidBook} for details. For grading in the symplectic homology, we assume that $2 c_1(TW) = 0,$ in which case the Grassmannian Lagrangian bundle $\cl Lag(M) \to M$ admits a cover $\til{\cl Lag}(M) \to M$ with fibers given by universal covers of the former fibers, and for signs we fix a background class $b \in H^2(W,\bF_2).$ In the main case we consider, $W = T^*L,$ for $L$ a closed manifold, our assumption holds, and we set $b = \pi^* w_2(L),$ for $\pi:T^*L \to L$ the natural projection. The existence of the cover can be deduced by considering the section of $\cl Lag(M)$ given by the Lagrangian subspaces tangent to the fibres. We equip each exact Lagrangian $L' \subset T^*L$ with the structure of a brane as follows. By \cite{Kragh-nearby}, the Maslov class of $L'$ vanishes, whence $\til{\cl Lag}(M)|_{L'}$ admits $H^0(L',\Z) = \Z$-worth of sections, which we call gradings, of which we pick one. By \cite{Abouzaid-nearbyMaslov}, $\pi_{L'}^* w_2(L) = w_2(L'),$ hence $L'$ is relatively Spin with respect to $b$, and futhermore out of the $H^1(L',\bF_2) = H^1(L,\bF_2)$ choices of a relative Spin structure we fix one, such that $L'$ endowed with these choices is Floer-theoretically equivalent to the zero-section $L$ (see Theorem \ref{thm: nearby Lag Fukaya}) with the standard relative Spin structure and grading. Throughout the paper, when considering Lagrangians, we keep in mind such an underlying determination of a brane structure.
%\red{add}

%\red{turning off slope}

\bs

Cycles in Deligne-Mumford moduli spaces of disks, considered as Riemann surfaces with boundary, decorated with interior and boundary punctures, whose universal curves are equipped with choices of positive or negative (input or output type) cylindrical ends at each puncture, induce operations on the various Floer homology groups considered. Indeed, we may equip the universal curves with Floer data compatible with gluing and compactification, wherein the cylindrical ends allow one to write suitable Floer equations and asymptotic conditions on the punctures to land in the correct Floer complexes. For more details we refer to \cite{SeidelSolomon-q}. In our case, as we wish to consider the behavior of actions and energies in our operations, we need to make further choices. In particular, we use the notion of cylindrical strips introduced and used in \cite{KS-bounds}. In fact, the Floer decorations for our main homological operation were already considered in \cite{KS-bounds} in the case of closed monotone symplectic manifolds, and their monotone Lagrangian submanifolds. We note that in constast to the closed case, in the case of Liouville manifolds one must ensure that the images of all Floer solutions lie in a compact subset of $W.$ This is accomplished by the integrated maximum principle (see \cite[Lemma 7.2]{AbouzaidSeidel} or \cite[Section 5.2.7]{AbouzaidBook}).

%\red{a bit more formulas for the TQFT}\\
%\red{TQFT: define BV-operator in SH}\\

In particular, consider the moduli space of disks with a unique input interior marked point and one output boundary marked point, with the cylindrical end at the interior marked point chosen so that the asymptotic marker points towards the boundary marked point. This moduli space is a point, and we may equip it with a choice of a cylindrical strip from the input to the output. We set the Hamiltonian Floer datum to be $H_\lambda \otimes dt$ on the cylindrical strip. Furthermore, we choose boundary condition $L$ for the Floer solutions. This gives us, for a suitable perturbation datum $\cD = \cD^{L,L}$ with Hamiltonian part compactly supported and of $C^2$ norm $o(1)$ as $\lambda \to \infty,$ an operation \[\phi_L^0: CF^*(H_{\lambda}) \to CF^*(L,L; H_{\lambda}) \to CF^*(L,L;\cD),\] which is a chain map. This operation yields the canonical restriction map \[r_L: SH^*(W) \to HF^*(L,L).\] As by \cite[Section 2.5]{KS-bounds}, arguing up to $\epsilon,$ the Floer data chosen as above has zero curvature, all perturbation data, in particular $\cD^{L,L},$ can be chosen to have Hamiltonian parts sufficiently small, so that this operation satisfies $\cA_{L,\cD}(\phi^0_L(x)) \leq \cA_{H_{\lambda}}(x) + 2 \epsilon.$

%\red{maybe write a general form of action estimates}

%check direction of action estimates everywhere; 

Furthermore, consider the moduli space of disks with two boundary marked points, an input and an output, and one interior marked point with asymptotic marker pointing towards the output. This moduli space is identified with an interval $\R,$ and its Deligne-Mumford compactification is identified with a closed interval by adding nodal disks at $-\infty,$ and $+\infty.$ See \cite{SeidelSolomon-q} for a description of these nodal disks. Choose cylindrical ends accordingly, and choose a cylindrical strip between the interior input and the boundary output. On this cylindrical strip, let the Hamiltonian part of the Floer datum be $H_{\lambda} \otimes dt.$  This yields an operation: 
\[ \phi^1_{L_0,L_1}: CF^*(H_{\lambda}) \otimes CF^*(L_0,L_1) \to CF^*(L_0,L_1; H_{\lambda})[-1] \to CF^*(L_0,L_1)[-1] . \] Considering the above compactification, one obtains \cite{SeidelSolomon-q} that $\phi^1_{L_0,L_1}$ provides a homotopy between the two maps \[\mu_2(\phi^0_{L_0}(a) , x),\] \[(-1)^{|a|\cdot |x|} \mu_2(x,\phi^0_{L_1}(a)),\] where $a \otimes x \in CF^*(H_{\lambda}) \otimes CF^*(L_0,L_1).$ Furthermore, by our choice of Floer data on the cylindrical strip, whose curvature vanishes by definition, choosing the Floer data $\cl D = \cl D^{L_0,L_1}$ to have sufficiently small Hamiltonian part we obtain that for all $a \otimes x \in CF^*(H_{\lambda}) \otimes CF^*(L_0,L_1),$ \[\cA_{L_0,L_1,\cl D}(\phi^1_{L_0,L_1}(a,x)) \leq \cA_{H_{\lambda}}(a) + \cA_{L_0,L_1,\cl D}(x) + 2\epsilon.\]

Finally, let $a \in CF^*(H_{\lambda})$ be a cycle, whose cohomology class represents $\alpha \in SH^*(W),$ with \[r_L(\alpha) = 0 \in HF^*(L,L).\] Following \cite[Definition 4.2]{SeidelSolomon-q}, we call $L$ {\it $a$-equivariant with primitive $c_L \in CF^*(L,L)$} if \[\phi^0_{L}(a) = \mu_1(c_L).\]  

As in \cite[Equation 4.4]{SeidelSolomon-q}, given a cycle $a \in CF^k(H_{\lambda}),$ and two $a$-equivariant Lagrangians $L_0,L_1 \subset D,$ with primitives $c_{L_0} \in CF^{k-1}(L_0,L_0),$ $c_{L_1} \in CF^{k-1}(L_1,L_1),$ we can upgrade $\phi^1_{L_0,L_1}(a,-)$ to a chain map \[\til{\phi}^1_{L_0,L_1}(a,-): CF^*(L_0,L_1) \to CF^*(L_0,L_1)[-1+k],\] by setting for homogeneous $x \in CF^*(L_0,L_1),$  \[\til{\phi}^1_{L_0,L_1}(a,x) = {\phi}^1_{L_0,L_1}(a,x) - \mu_2(c_{L_0},-) + (-1)^{(k-1) |x|} \mu_2(-,c_{L_1}).\] Of course $\til{\phi}^1_{L_0,L_1}(a,x)$ depends on the choice of primitives $c_{L_0}, c_{L_1}$ for the $a$-equivariant structures. We keep this dependence implicit in the notation for conciseness, and we further discuss it in Section \ref{sec: homology}.

%\red{fix signs; do we need to assume that $k$ is odd?}

For sufficiently $C^2$ Hamiltonian-small perturbation data $\cD^{L_0,L_0}, \cD^{L_1,L_1}$ it is easy to see that all chains in $CF^*(L_0,L_0), CF^*(L_1,L_1)$ are of actions $\cA_{L_0,L_0; \cD}, \cA_{L_1, L_1; \cD}$ bounded in absolute value by $\epsilon.$ Hence, assuming that $\cA_{H_{\lambda}}(a) \geq \epsilon_{\alpha}/2 \gg \epsilon,$ we obtain that \begin{equation}\label{eq: action of upgraded multiplication map}\cA_{L_0,L_1,\cl D}(\til{\phi}^1_{L_0,L_1}(a,x)) \leq \cA_{L_0,L_1,\cl D}(x) + \cA_{H_{\lambda}}(a) + 2\epsilon.\end{equation} By our choices in the construction of symplectic cohomology, the assumption on $\cA_{H_{\lambda}}(a)$ is verified unless $a$ lies in the subcomplex formed by generators located in $D.$ In case $a$ does lie in this subcomplex, and $a \neq 0,$ then $\cA_{H_{\lambda}}(a) \geq -\epsilon,$ in which case \eqref{eq: action of upgraded multiplication map} holds still. The case $a = 0$ is trivial. Hence \eqref{eq: action of upgraded multiplication map} applies to all $a \in CF^*(H_{\lambda}),$ and $x \in CF^*(L_0,L_1).$

%In our applications, this will be the case for homological reasons. 

%\red{write $\phi^0, m_{L_0,L_1}, \phi^1_{L_0,L_1},
% \til{\phi}^1_{L_0,L_1}$}
\bs

%\red{Careful with orientations etc} 
Finally, we recall two fundamental results on the symplectic topology of cotangent bundles. The first result, proved by Viterbo \cite{Viterbo-iso}, Abbondandolo-Schwarz \cite{AbbSchwarz}, and Salamon-Weber \cite{SalamonWeber} in the case of $\bK = \bF_2,$ or for $Spin$ manifolds and arbitrary coefficients, and by Abouzaid and Kragh in the geneneral case (see \cite{AbouzaidBook} and the references therein), asserts a relation between the symplectic cohomology of $W = (T^*L, \theta_{\mrm{can}})$ considered as a Weinstein manifold, and the homology of the free loop space $\cL L.$ In general, to compare signs between the two theories, a local system on $\cL T^*L$ should be introduced, as mentioned above. For certain choices of $L$ and $\bK,$ such as $\bK = \bF_2,$ or $L$ being Spin, this local system is trivial and can therefore be ignored. Finally, we note that it is important for our purposes, that this is an isomorphism of Gerstenhaber algebras over $\bK$, rather than simply one of $\bK$-algebras. This aspect of the isomorphism is discussed in \cite{AbouzaidBook}: in fact it is an isomorphism of BV-algebras. 

\begin{thm}[Viterbo isomorphism]\label{thm: Viterbo isomorphism}
Let $W = T^*L$ with the standard Liouville structure. There exists an isomorphism \[\Phi: SH^{\ast}(W) \to H_{n-\ast}(\cL L),\] of BV-algebras over $\bK,$ where $SH^{\ast}(W)$ is endowed with the pair-of-pants product, and the BV-operator arising from the moduli space of cylinders with free asymptotic markers at infinity, while $H_{n-\ast}(\cL L)$ is endowed with the Chas-Sullivan product, and the BV-operator given by suspending the $S^1$-action by loop-rotation on $\cl LL$. The map $ev:H_{n-\ast}(\cl LL) \to H_{n-*}(L)$ corresponds to the map $r_L: SH^*(W) \to HF^*(L)$ by the isomorphism $\Phi$ and Poincar\'{e} duality.
\end{thm}

The second result, due to Fukaya-Seidel-Smith \cite{FukayaSeidelSmith-cotangentsc} in the simply connected case, and Abouzaid \cite{Abouzaid-nearbyMaslov} and Kragh \cite{Kragh-nearby}, in the general case, asserts that each exact Lagrangian $L'$ in the cotangent bundle $T^*L$ is isomorphic to $L$ in the Fukaya category of $T^*L.$ It is not difficult to observe that this isomorphism is in fact an isomorphism of modules over $SH^*(T^*L)$: indeed, after one knows that the isomorphism is given by multiplication by continuation elements, this is a consequence of the homotopy property of $\phi^1_{L_0,L_1}.$ We state a simplified version that is sufficient for our purposes, referring to \cite{AK-simplehomotopy}. %\red{degrees above and below}

\begin{thm}[Exact nearby Lagrangians are Floer-theoretically equivalent]\label{thm: nearby Lag Fukaya}
Let $L$ and $\bK$ be as above, and let $L'$ be an exact Lagrangian in $T^*L.$ Then there exists an integer $i = i_{L'} \in \Z$ such that for each exact Lagrangian $K,$ the $SH^*(T^*L)$-modules $HF^*(L',K),$ and $HF^{*+i}(L,K)$ are isomorphic, and the same is true for $HF^*(K,L'),$ and $HF^{*-i}(K,L),$ for certain $i \in \Z.$ The chain-level quasi-isomorphisms in both directions can be taken to be multiplication operators \[\mu_2(-,x): CF^*(L',K) \to CF^{*+i}(L,K),\] \[\mu_2(-,y): CF^*(L,K) \to CF^{*-i}(L',K),\] respectively \[\mu_2(y,-): CF^*(K,L') \to CF^{*-i}(K,L),\] \[\mu_2(x,-):CF^*(K,L) \to CF^{*+i}(K,L'),\] for certain cycles $x \in CF^{i}(L,L'),$ $y \in CF^{-i}(L',L).$ 
\end{thm}

%\red{define spectral norm}\\
%It follows that for two exact Lagrangians $L_0,L_1$ in $T^*L,$ a suitable composition of the above quasi-isomorphisms is a quasi-isomorphism \[\psi_{L_0,L_1} : CF^*(L_0,L_0) \to CF^*(L_1,L_1).\] 

At this point we define the spectral norm $\gamma(L_0,L_1)$ for exact Lagrangians in $T^*L$ as follows. Choose primitives $f_0,f_1$ of the restrictions $\theta|_{L_0}, \theta|_{L_1}$ of the Liouville form $\theta$ to $L_0,L_1$ respectively. This allows us to filter $CF^*(L_0,L_1)$ by an action functional induced by $\ul L_0 = (L_0,f_0),$ $\ul L_1 = (L_1,f_1).$ Since $HF^*(L_0,L_1) \cong HF^*(L,L) \cong H^*(L),$ consider the classes $\mu,e \in HF^*(L_0,L_1)$ that correspond to the generator $\mu_L = PD([pt]) \in H^n(L),$ and the unit $1 = PD([L]) \in H^0(L)$ respectively. Recall that for a class $a \in HF^*(L_0,L_1) \setminus \{0\},$ we defined the Lagrangian spectral invariant as
\[c(a,\ul L_0, \ul L_1;\cD) = \inf\{t \in \R\;|\; a \in \ima\left(HF^*(L_0,L_1; \cD)^{<t} \to HF^*(L_0,L_1\; \cD)\right) \}.\] 
These invariants are finite, and satisfy numerous useful properties, %(see for example \cite{S-Zoll} for a relevant discussion). For example they are Lipschitz in the Hofer norm with respect to Hamiltonian deformations of $L_0$ and $L_1,$ and hence 
and in particular they are defined for arbitrary exact $L_0,L_1,$ by taking the limit as the Hamiltonian term in the perturbation datum $\cD$ goes to zero. We set the spectral norm to be \begin{equation}\label{eq: spec norm}\gamma(L_0,L_1) = c(\mu,\ul L_0,\ul L_1) - c(e,\ul L_0, \ul L_1).\end{equation} Note that as a difference of two spectral invariants it does not depend on the choice of enhancements $\ul L_0, \ul L_1$ of $L_0,L_1.$ Furthermore, by considering the identity $\mu_L = \mu_L \ast 1$ in $H^*(L),$ one obtains the identity $\mu = \mu_{L_1} \ast e,$ under the isomorphisms $H^*(L) \cong HF^*(L,L) \cong HF^*(L_1,L_1) \cong H^*(L_1)$ and $H^*(L) \cong HF^*(L,L) \cong HF^*(L_0,L_1),$ from which one obtains that $\gamma(L_0,L_1) \geq 0,$ and that the inequality is strict unless $L_0=L_1$ (see \cite{KS-bounds}). Further properties of spectral invariants imply that $\gamma(L_0,L_1) = \gamma(L_1,L_0)$ for all $L_0,L_1$ exact, and that $\gamma(L_0,L_1) \leq \gamma(L_0,K) + \gamma(K,L_1)$ for all $L_0,L_1,K$ exact, whence $\gamma$ defines a metric on the space of exact Lagrangian submanifolds of $T^*L.$ Furthermore, this metric is invariant under the action of the group of Hamiltonian diffeomorphisms: for all $H \in C^{\infty}_c([0,1] \times T^*L, \R)$ and $L_0,L_1$ exact, $\gamma(\phi L_, \phi L_1) = \gamma(L_, L_1),$ where $\phi = \phi^1_H,$ is the time-one map of the Hamiltonian isotopy generated by $H.$ Finally, note that we shall study the restriction of $\gamma$ to the subspace of exact Lagrangian submanifolds in $D \subset T^*L,$ where $D$ is a bounded Liouville domain with completion $T^*L.$ 
%\red{TQFT: define BV-operator in SH}\\
%\red{define BV-operator in string topology}\\
%\red{should one prove Theorem C?}\\

\section{Proofs}\label{sec: proof}

\subsection{A homological calculation.}\label{sec: homology}

We start with a new calculation of the map \[\til{\phi}^1_{(L_0,L_1),a}(-) = \til{\phi}^1_{L_0,L_1}(a,-)\] on the level of homology in the setting of cotangent bundles. Let $\alpha \in SH^k(W)$ for $W = T^*L$ be such that $r_L(\alpha) = 0$ in $HF^*(L,L).$ Then, in view of Theorem \ref{thm: nearby Lag Fukaya}, for each closed exact Lagrangian $L_0$ in $W$ we have again $r_{L_0}(\alpha) = 0$ in $HF^*(L_0,L_0).$ Now consider two such exact Lagrangians $L_0,L_1$ and a cycle $a \in CF^k(H_{\lambda})$ representing $\alpha.$ We require a compatibility assumption on the choice of $a$-equivariant primititives $c_{L_0},$ $c_{L_1}.$ It is immediate from the definition that $c_{L_i},$ $i \in \{0,1\},$ are defined uniquely up to closed chains in $CF^{k-1}(L_i, L_i).$ Furthermore $\til{\phi}^1_{L_0,L_1}(a,-)$ up to chain homotopy depends only on $c_{L_i}$ up to exact chains in $CF^{k-1}(L_i, L_i).$ Adding a cycles $c_i \in CF^{k-1}(L_i,L_i)$ to $c_{L_i}$ changes $\til{\phi}^1_{L_0,L_1}(a,x)$ on the homology level by \[ -\mu_2([c_0],x) + (-1)^{(k-1)|x|}\mu_2(x,[c_1]).\] In particular, by Theorem \ref{thm: nearby Lag Fukaya}, and the graded-commutativity of $HF^*(L,L)$ as an algebra, if $L_0 = L_1 = L,$ and we choose $c_{L_0} - c_{L_1}$ exact, then we get a canonically defined operator \[P'_a = [\til{\phi}^1_{{(L,L)},a}]\] on homology.
%
%\red{Pick an equivariant primitive $c_{L_0} \in CF^{k-1}(L_0,L_0)$ for $a$ on $L_0,$ and choose an equivariant primitive $c_{L_1} \in CF^{k-1}(L_1,L_1)$ for which $\psi_{L_0,L_1}(c_{L_0}) - c_{L_1}$ is exact. For example, observing that \[{\phi}^0_{L_1}(a) = \psi_{L_0,L_1}({\phi}^0_{L_0}(a)) + \mu_1(c_{L_0,L_1})\] for certain $c_{L_0,L_1} \in CF^{k-1}(L_1,L_1)$ we can take $c_{L_1} = \psi_{L_0,L_1}(c_{L_0}) + c_{L_0,L_1}.$ 
		
%With this compatible choice of equivariant structures, set $\til{\phi}^1_{(L_0,L_1),a}(-): = \til{\phi}^1_{L_0,L_1}(a,-)$. 

In view of the observations above, by means of Theorem \ref{thm: nearby Lag Fukaya}, if we are merely interested in the operator $P'_{(L_0,L_1),a}= [\til{\phi}^1_{(L_0,L_1),a}]$ on the homological level, and $L_0, L_1$ are exact Lagrangians in $T^*L,$ then we may replace both $L_0,$ and $L_1$ by $L$ in the definition. Furthermore, we may choose the primitives $c_{L_0}, c_{L_1}$ in such a way that the corresponding operator is $P'_a =[\til{\phi}^1_{{(L,L)},a}]$ as defined above with respect to the same $a$-equivariant primitive on both copies of $L.$ This point, requires a somewhat more careful discussion of the moduli spaces involved, despite being intuitively clear. To streamline the exposition, we therefore fomulate it below and defer its proof to the end of the section. 

\begin{prop}\label{prop: invariance of phi^1}
Given $a \in CF^{k}(H_{\lambda}),$ and $a$-equivariant exact Lagrangians $L_0, L_1$ in $T^*L,$ there exist equivariant primitives $c_{L_0}, c_{L_1}$ for $L_0, L_1$ such that under the isomorphism \[HF^*(L_0,L_1) \cong HF^{*+i_{L_0} - i_{L_1}}(L,L)\]  the operation $[\til{\phi}^1_{(L_0,L_1),a}]$ corresponds to the canonical operator \[P'_a: HF^*(L,L) \to HF^{*-1+k}(L,L).\]
\end{prop}

We proceed to compute the latter operation $P'_a$ as follows.

%We compute the operation $P'_a = [\til{\phi}^1_{{(L,L)},a}],$ with both copies of $L$ endowed with the same $a$-equivariant primitive, as follows.

% was computed by Seidel and Solomon \cite[Remark 4.4]{SeidelSolomon-q}: by means of Theorem \ref{thm: Viterbo isomorphism}, and its proofs, under natural isomorphisms the operator $P'_a$ corresponds to $P_a = ev\circ \Delta \circ m_a \circ \iota$ from Definition \ref{def: pt invertible}. %\red{state as theorem, and prove; will require identifying various other operations}

\begin{prop}\label{thm: product map isomorphism}
%The isomorphisms $\Phi: SH^{\ast}(T^*L) \to H_{n-\ast}(\cL L)$ identifies $P'_a: SH^{\ast}(T^*L) \to SH^{\ast}(T^*L)$ and $P_a: $

The isomorphism $HF^*(L,L) \cong H_{n-*}(L)$ obtained from by the isomorphism $HF^*(L,L) \cong H^*(L)$ followed with the Poincar\'{e} duality isomorphism $H^*(L) \cong H_{n-\ast}(L),$ identifies $P'_a: HF^*(L,L) \to HF^*(L,L)$ and the map $P_a: HF_{n-\ast}(L) \to HF_{n-\ast}(L).$
\end{prop}  

%\red{prove Theorem D in more detail}

%
%Note that outside $D \cup C_{\lambda},$ $X_{H_{\lambda}}$ is tangent to the contact-type foliation, hence maximum principle \cite[Lemma 5.5]{KhovanovSeidel} applies still, to keep the corresponding solutions to Floer equations within $D \cup C_{\lambda}.$

\begin{proof}[Proof of Proposition \ref{thm: product map isomorphism} (sketch)]

One way to prove this result involves first showing that $P'_a = r_L \circ m'_a \circ \iota',$ where $\iota'$ is a homological inclusion map $HF^*(L,L) \cong HF_{n-\ast}(L,L) \to SH^*(T^*L),$  $m'_a: SH^*(T^*L) \to SH^*(T^*L)$ is the right symplectic homology bracket with $a \in SH^*(T^*L),$ given by the pair of pants product and the BV-operator $\Delta': SH^*(T^*L) \to SH^*(T^*L)[-1],$ and $r_L: SH^*(T^*L) \to HF^*(L,L)$ is the restriction map. Consequently, one shows that each of these maps is identified with $\iota, m_{\Phi(a)}, \Delta,$ and $ev,$ respectively under the isomorphisms $\Phi$ and $HF^*(L,L) \cong H^*(L) \cong H_{n-\ast}(L).$ Theorem \ref{thm: Viterbo isomorphism} takes care of identifying $m_{a'}$ and $m_{\Phi(a)},$ and $\Delta'$ with $\Delta.$ It is left to identify $\iota'$ with $\iota,$ and $r_{L}$ with $ev.$ The map $r_{L}$ in the latter pair is a suitable closed-open map, and the identification is carried out by following the isomorphism $\Phi$ from \cite{AbouzaidBook}. The map $\iota'$ defined below can again be seen to correspond to $\iota$ by following the construction in \cite{AbouzaidBook}. Alternatively, see \cite{AbbSchwarz,AbbSchwarz-corr}.

% wherein $H_{\lambda}$ is small. Consequently action estimates \cite{AbbSchwarz} readily apply.

To show that $P'_a = r_L \circ m'_a \circ \iota',$ we first observe that $\iota'$ is given \cite{AbouzaidBook} (see also \cite{AbbSchwarz}) by the moduli space of disks with one interior marked point, which is an output, and boundary conditions on a cotangent fibre $T^*_x L$ where $x,$ that we consider to be an input, is allowed to vary freely in $L.$ We choose the perturbation datum to coincide with $H_{\lambda} \otimes dt$ on a cylindrical end by the output, and with zero near the boundary. This is possible to achieve while keeping the conditions of the integrated maximum principle of Abouzaid and Seidel \cite{AbouzaidSeidel}, that keeps the corresponding solutions to the Floer equations within $D \cup C_{\lambda}.$  

Similarly, $r_L$ is given by the moduli space of disks with one interior marked point, which is an input, and one boundary marked point which is an output, with the asymptotic marker at the interior marked point pointing towards the boundary marked point, and with boundary conditions on $L.$ It is then easy to see by gluing, that $r_L \circ \iota'$ is given by the moduli space of annuli with boundary conditions on a cotangent fiber on one boundary component and $L$ on the other, with a boundary marked point, an output: that is in a concrete model, if the annulus is given by $[0,l] \times S^1,$ the boundary marked point is $(l,\zeta)$ for $\zeta \in S^1,$ and the boundary condition at $\{0\} \times S^1$ is on $T^*_x L$ for a varying $x.$ Moreover it is easy see, for example by a suitable homotopy of the Hamiltonian perturbation data and a pearly model for this operation, that on the homology level $r_L \circ \iota' = \id.$ Hence $P'_a =  P'_a \circ r_L \circ \iota'.$ We claim that it is therefore given by the compactified moduli space of annuli with one boundary component marked by a varying cotangent fibre, the other boundary component marked by $L$ and endowed with an unconstrained boundary marked point, an output, and one interior input marked point. The asymptotic marker at the interior marked point is pointing towards the output, and we use the interior marked point to plug in $a.$ 

This last claim requires further elaboration. We proceed as follows. First we glue the moduli space of disks describing $\phi^1_{L,L}(a,-)$ and the moduli space of disks describing $r_L \circ \iota'$ on the chain level. After fixing parametrization, we can fix the former to be given by the standard disk $D^2 \subset \C$ with two fixed marked points $\zeta_{-} = -1, \zeta_{+} = 1$ on the boundary, and an interior marked point $z = iy$ with $\mrm{Re}(z) = 0,$ and asymptotic marker pointing towards $\zeta_{+}$ along a hyperbolic geodesic. Hence the glued operation is given by an cylinder $[0,l] \times S^1$ with boundary conditions as above, and interior marked point constrained to an arc $\gamma_{1}$ with boundary $\{(l,\zeta_1)\}-\{(l,\zeta_0)\}$ on ${l} \times S^1$ separating $(l,0)$ and $\{0\} \times S^1,$ with asymptotic marker pointing towards $(l,0).$ Furthermore, we identify, up to chain-homotopy, the correction term $- \mu_2(c_{L},-) + (-1)^{(k-1) |x|} \mu_2(-,c_{L})$ in $\til{\phi}^1_{L,L,a},$ precomposed with the chain-level map $\psi$ giving $r_L \circ \iota',$ as the operator given by the same Riemann surface, now considered as an annulus with outer boundary $\{l\} \times S^1$, with the same boundary conditions and boundary marked points, except that the interior marked point, used to plug in $a,$ is now constrained to an arc $\gamma_2$ with boundary $-\{(l,\zeta_1)\}+\{(l,\zeta_0)\},$ which now does not separate $(l,0)$ and $\{0\} \times S^1$ (this is a result of parametric gluing applied to a one-parametric family of nodal annuli, consisting of an annulus and a disk related by a node $(l,\zeta),$ with the interior marked point in the disk, used to plug in $a,$ the asymptotic marker pointing at the node, interpolating between the surface for $(l,\zeta_1),$ and the one for $(l,\zeta_0),$ with $\zeta$ in the spherical arc $[\zeta_1,\zeta_0]$ not containing $0$). The asymptotic marker at the interior marked point still points towards $(l,0).$  Choosing $\gamma_1$ and $\gamma_2$ suitably, and considering a homotopy of the decorated Riemann surfaces corresponding to the loop $\gamma_1 \# \gamma_2,$ we obtain the claim. %\red{change to boundary + parametric gluing}

%with a suitable homotopy operator \red{describe more})

%Introducing a pearly model for this operation, coming from Morse functions on $L$ (there are two $s$-dependent Morse functions, $f_{-,s},$ $s \in (-\infty,0],$ $f_{+,s},$ $s \in [0,\infty),$ with $f_{\pm,s} \equiv f$ for $|s| \gg 1,$ for a fixed Morse function $f$ on $L$). We may achieve transversality by a generic choice of such functions, and since we may turn off $H_{\lambda}$ keeping compactness, we may assume to have {\em $J$-holomorphic} annuli, which must then be constant by exactness of $L,$ and the cotangent fibers (by Stokes' theorem). 

Finally, it is easy to see, by gluing again, that $r_L \circ m'_a \circ \iota'$ is given by a homotopic moduli space of decorated annuli. This is immediate by gluing from the description of the string bracket \cite[Section 2.5.1]{AbouzaidBook} as the moduli space of spheres with 3 marked points, where in the model of $S^2 = \C P^1 = \C \cup \{\infty\}$ we choose the marked points $z_0 = \infty,$ $z_1 = 0,$ and $z_2$ restricted to the unit circle $z_2 = e^{i\theta}$ for $\theta \in \R/(2\pi)\Z.$ The asymptotic markers at $z_0, z_1$ are chosen to be tangent to the imaginary axis, and pointing in the negative, resp. positive, direction. The asymptotic marker at each $z_2 = e^{i\theta}$ is chosen by identifying $S^2\setminus\{z_0,z_2\}$ biholomorphically with $\R \times S^1,$ so that the output asymptotic marker at $z_0$ correspond to the ray $(-\infty,0) \times \{0\}$ in the negative end $(-\infty,0) \times S^1,$ and choosing the input asymptotic marker at $z_2$ to correspond to the ray $(0,\infty) \times \{0\}$ in the positive end $(0,\infty) \times S^1.$
%The most interesting point is that $r_L \circ \Delta'$ is given by the moduli space of disks with one input interior marked point, and one output boundary marked point, with the asymptotic marker at the interior marker is now allowed to {\em vary freely}, and with boundary conditions on $L.$ By choosing reparametrized coordinates where the interior marked point is sent to $0 \in D$ in the disk $D,$ and the asymptotic marker to $1 \in T_0 D,$ and gluing it with the decorated Riemann surfaces giving $m'_a$ and $\iota',$ we deduce the required statement.
\end{proof}

Using similar arguments, we now prove Proposition \ref{prop: invariance of phi^1}. 

\begin{proof}[Proof of Proposition \ref{prop: invariance of phi^1}]
We first remark that in the case that $L_0,L_1$ are Hamiltonian isotopic to the zero section, the proof is just a matter of a standard homotopy argument. Therefore we sketch the proof in the case of arbitrary exact $L_0, L_1.$ 

Firstly consider the homotopy inverse quasi-isomorphisms $\psi: CF^*(L_0,L_1) \to CF^{*+j}(L,L)$ and $\psi': CF^*(L,L) \to CF^{*-j}(L_0,L_1)$ for $j = i_{L_0} - i_{L_1},$ \[\psi = \mu_2(y_1,-) \circ \mu_2(-,x_0),\] \[\psi' = \mu_2(x_1,-) \circ \mu_2(-,y_0),\] for classes $x_0 \in CF^{i_{L_0}}(L,L_0), y_0 \in CF^{-i_{L_0}}(L_0,L)$ and $x_1 \in CF^{i_{L_1}}(L,L_1), y_1 \in CF^{-i_{L_1}}(L_1,L)$ provided by Theorem \ref{thm: nearby Lag Fukaya}. We observe that it is enough to prove that the map \[\psi \circ \til{\phi}_{(L_0,L_1),a} \circ \psi': CF^*(L,L) \to CF^{*-1+k}(L,L)\] for arbitrarily chosen primitives $c'_{L_0},$ $c'_{L_1}$ for the $a$-equivariant structure on $L_0,L_1,$ is chain homotopic to \[P'_a + \mu_2(c,-)\] for a certain cycle $c \in CF^{-1+k}(L,L).$ Indeed in this case we can choose new primitives $c_{L_0} = c'_{L_0} + c_0,$ $c_{L_1} = c'_{L_1}$ for a cycle $c_0 \in CF^{k-1}(L_0,L_0)$ whose image under the quasi-isomorphism $\mu_2(y_0,-) \circ \mu_2(-,x_0): CF^*(L_0,L_0) \to CF^*(L,L)$ of Theorem \ref{thm: nearby Lag Fukaya} represents $[c]$ in homology. Therefore by an argument involving gluing and homotopies of families of decorated Riemann surfaces, and a relation between operations obtained by restricting the interior marked point to lie on a suitable collection of four arcs in the disk with five boundary inputs and one boundary output, we deduce, choosing an arbitrary primitive $c_L$ for the $a$-equivariant structure on $L$ that the operator $ \psi \circ \til{\phi}^1_{(L_0,L_1),a} \circ \psi'$ is chain-homotopic to the sum \[ \til{\phi}^1_{(L,L),a} + \mu_2( \mu_2(y_1,\til{\phi}^1_{(L,L_1),a}(x_1)) - \mu_2(\til{\phi}^1_{(L_0,L),a}(y_0),x_0) , - ),\] which is of the desired form. 
\end{proof}

\subsection{Proof of Theorem \ref{thm:Vit-prod}}\label{subsec: proof main}
%\red{degrees}

Let $a_1,\ldots, a_N \in H_{n-\ast}(\cl LL)^+$ be such that \begin{equation}\label{eq: product identity} P_{a_N} \circ \ldots \circ P_{a_1}([pt]) = [L].\end{equation} Let $\mu, e \in HF^*(L_0, L_1)$ be such that $\mu$ corresponds to $[pt]$ and $e$ corresponds to $[L]$ under the isomorphism  \[HF^{*+i_1-i_0}(L_0,L_1) \xrightarrow{\simeq} HF^{*-i_0}(L_0,L) \xrightarrow{\simeq} HF^*(L,L) \cong H_{n-\ast}(L),\] for suitable integers $i_0, i_1 \in \Z.$ In this case the spectral norm $\gamma(L_0,L_1)$ is given by \[\gamma(L_0,L_1) = c(\mu,L_0,L_1) - c(e, L_0, L_1).\] It is therefore sufficient to prove that there exists a constant $C(g,L; \bK)$ such that \[c(\mu,L_0,L_1) \leq c(e,L_0,L_1) + C(g,L;\bK).\] Let $a'_1 = \Phi^{-1}(a_1),\ldots, a'_N = \Phi^{-1}(a_N).$ Since $a_1,\ldots, a_N \in H_{n-\ast}(\cl LL)^+,$ we obtain that $L_0, L_1$ are $a_j$-equivariant for all $1 \leq j \leq N.$ In view of Proposition \ref{thm: product map isomorphism}, the identity \eqref{eq: product identity} corresponds to the identity \[P'_{(L_0,L_1),a'_N} \circ \ldots P'_{(L_0,L_1),a'_1} (e) = \mu.\] Set $C_j = \cA_{H_{\lambda}}(\til{a}'_j)$ for representatives $\til{a}'_j$ of $a'_j,$ and $x_j = P'_{(L_0,L_1),a'_j} \circ \ldots P'_{(L_0,L_1),a'_1} (e),$  $1\leq j \leq N,$ $x_0 = e.$ We lift these elements to the chain level as follows. Let $\til{x}_0 = \til{e} \in CF^*(L_0,L_1)$ be a chain representative of $e$ with \begin{equation}\label{eq: def of e tilde}\cA_{L_0,L_1,\cl D}(\til{e}) \leq c(e,L_0,L_1) + \epsilon\end{equation} (where $\cl D$ is chosen to be Hamiltonian-small). Then for $1 \leq j \leq N$ we set \[\til{x}_j = \til{\phi}^1_{(L_0,L_1),\til{a}'_j} \circ \ldots \circ \til{\phi}^1_{(L_0,L_1),\til{a}'_1}.\] In this situation we obtain by \eqref{eq: action of upgraded multiplication map} that for all $0 \leq j < N,$ \[\cA_{L_0,L_1,\cl D}(\til{x}_{j+1}) \leq \cA_{L_0,L_1,\cl D}(\til{x}_j) + \cA_{H_{\lambda}}(a_{j+1}) + 2\epsilon.\] Therefore by \eqref{eq: def of e tilde}, we obtain that \[\cA_{L_0,L_1,\cl D}(\til{x}_N) \leq c(e,L_0,L_1) + \sum_{j=1}^{N} C_j + 2(N+1)\epsilon.\] As $[\til{x}_N] = x_N = \mu,$ by definition of the spectral invariant we have \[c(\mu,L_0,L_1) \leq \cA_{L_0,L_1,\cl D}(\til{x}_N),\] and this finishes the proof.

\begin{rmk}
The inequality \[c(\mu,L_0,L_1) \leq c(e,L_0,L_1) + \sum_{j=1}^{N} C_j + 2(N+1)\epsilon\] finishing the proof can be optimized as follows. First we may choose $C_j = c([a_j], H_{\lambda}),$ to obtain \[c(\mu,L_0,L_1) \leq c(e,L_0,L_1) + \sum_{j=1}^{N} c(a'_j, H_{\lambda}) + 3(N+1)\epsilon.\] Furthermore, as $\lambda \to \infty,$ by definition of spectral invariants in filtered symplectic homology, we obtain that $c(a'_j, H_{\lambda}) \to c(a'_j, D,S).$ Therefore we can write \[c(\mu,L_0,L_1) \leq c(e,L_0,L_1) + \sum_{j=1}^{N} c(a'_j, D,S) + 3(N+1)\epsilon,\] and sending $\epsilon$ to $0,$ we finally get the bound \begin{equation}\label{eq: bound on spectral norm}
\gamma(L_0,L_1) \leq \sum_{j=1}^{N} c(a'_j, D,S).
\end{equation} 
%We explain in \ref{subsec: sharpness} that this bound is in fact sharp in certain cases.

\end{rmk}
%
%\subsection{On sharpness}
%
%We claim that the bound \eqref{eq: bound on spectral norm} is sharp in the case of $L = T^n$ with the standard Riemannian metric.
%It is easy to construct a Hamiltonian deformation of the zero

\subsection{Proof of Proposition \ref{prop: prod}} \label{subsec: proof prod} 

%\red{modify!}

Let $L,$ $L'$ be string point-invertible. We will show that so is $L \times L'.$ Let $a_1,\ldots,a_N,$ and $a'_1,\ldots,a'_{N'}$ be sequences exhibiting the fact that $L,$ respectively $L',$ are string point-invertible. Consider $b_k = a_k \otimes \iota([L']),$ for $1 \leq k \leq N,$ and $b_k = \iota([L]) \otimes a'_{k- N},$ for $N < k \leq N+N'.$ We claim that the sequence $b_1,\ldots, b_{N+N'},$ up to sign, is the required one for $L \times L'.$ Since it is indeed enough to consider the question up to signs, we will ignore signs coming from the Koszul rule for tensor products. Note that by K\"{u}nneth theorem $H_*(L \times L') \cong H_*(L) \otimes H_*(L')$ and $H_*(\cl LL \times \cl LL') \cong H_*(\cl LL) \otimes H_*(\cl LL').$ Furthermore, under changing the grading to $*-n$ everywhere, these splittings are tensor products of graded algebras. Moreover the maps $\iota$ and $ev$ commute with this tensor product decomposition. The BV-operator behaves in the following way: for homogeneous elements $x \in H_*(\cl LL),$ $x' \in H_*(\cl LL'),$ %\red{are signs correct? take care with $\ast$ vs $n-\ast$, also with signs for swaps in tensor product} 
\[\Delta_{L \times L'}(x \otimes x') = \Delta_L(x) \otimes x' + (-1)^{|x|} x \otimes \Delta_{L'}(x').\] This implies that for $a \in H_*(L),$ $a' \in H_*(L'),$ \[m_{a \otimes a'} (x \otimes x') = \pm m_a(x) \otimes a' \ast x' \pm  a \ast x \otimes m_{a'}(x').\]

Observe that $[pt_{L \times L'}] = [pt_L] \otimes [pt_{L'}],$ $[L \times L'] = [L] \otimes [L'],$ and furthermore for $a' = \iota([L']),$ $m_{a'} = 0.$ Therefore \[m_{b_1} \iota([pt_L] \otimes [pt_{L'}]) = m_{a_1 \otimes \iota([L'])} \iota([pt_L] \otimes [pt_{L'}]) = \pm m_{a_1}(\iota([pt_L])) \otimes \iota([pt_{L'}]).\]

%We compute \[b_1 \ast \iota([pt_L] \otimes [pt_{L'}]) = (a_1 \otimes [L']) \ast \iota([pt_L]) \otimes \iota([pt_{L'}]) = \pm (a_1 \ast \iota([pt_L])) \otimes \iota([pt_{L'}]).\] Since $\Delta \circ \iota = 0,$ we therefore obtain \[\Delta((a_1 \ast \iota([pt_L])) \otimes \iota([pt_{L'}])) = \Delta((a_1 \ast \iota([pt_L]))) \otimes \iota([pt_{L'}]).\] 

Finally, since $ev$ is an algebra map, and it commutes with the K\"{u}nneth decompositions, we obtain \[P_{b_1}([pt_{L\times L'}]) = \pm P_{a_1}([pt_L]) \otimes [pt_{L'}].\] The same argument shows that for all $1 \leq k \leq N,$ \[P_{b_k} \circ \ldots \circ P_{b_1}([pt_{L\times L'}]) = \pm P_{a_k} \circ \ldots \circ P_{a_1}([pt_L]) \otimes [pt_{L'}].\] In particular, for $k=N,$ \begin{equation}\label{eq: tensor 1} P_{b_N} \circ \ldots \circ P_{b_1}([pt_{L\times L'}]) = \pm [L] \otimes [pt_{L'}].\end{equation} 

%Now we calculate $P_{b_{N+1}} ([L] \otimes [pt_{L'}]).$ First \[(\iota([L]) \otimes a'_1) \ast (\iota([L]) \otimes \iota([pt_{L'}])) = \pm \iota([L]) \otimes (a'_1 \ast \iota([pt_{L'}])). \] Consequently \[\Delta(\iota([L]) \otimes (a'_1 \ast \iota([pt_{L'}]))) = \pm \iota([L]) \otimes \Delta (a'_1 \ast \iota([pt_{L'}])),\] whence \[P_{b_{N+1}} ([L] \otimes [pt_{L'}]) = \pm [L] \otimes P_{a'_1}([pt_{L'}]).\] 

Similarly, we obtain that for all $1 \leq k \leq N',$ \[P_{b_{N+k}} \circ \ldots \circ P_{b_{N+1}} ([L] \otimes [pt_{L'}]) = \pm [L] \otimes P_{a'_k} \circ \ldots \circ P_{a'_1}([pt_{L'}]).\] In particular, for $k = N',$ \begin{equation}\label{eq: tensor 2}P_{b_{N+N'}} \circ \ldots \circ P_{b_{N+1}} ([L] \otimes [pt_{L'}]) = \pm [L] \otimes [L'].\end{equation}
Hence, by \eqref{eq: tensor 1} and \eqref{eq: tensor 2}, we obtain \[P_{b_{N+N'}} \circ \ldots \circ P_{b_{1}} ([pt_{L \times L'}]) = \pm [L \times L'].\] If necessary, changing the sign of $b_{N+N'},$ we finish the proof. 

%\section{Auxiliary proofs.}

%\subsection{Proof of Theorem \ref{thm: gamma closed}}

%\section{Discussion}

%Hence $P_{b_1}([pt_{L \times L'}]) = ev(\Delta(b_1 \ast \iota([pt_L]) \otimes \iota([pt]))$

\section*{Acknowledgements}
I thank Mohammed Abouzaid for suggesting to look at equivariant Lagrangians in the quantitative context, and Leonid Polterovich for motivating discussions on applying symplectic cohomology. I thank Paul Biran, Lev Buhovsky, Octav Cornea, Michael Entov, Helmut Hofer, Yanki Lekili, Alexandru Oancea, James Pascaleff, Dmitry Tonkonog, and Sara Venkatesh for useful conversations. This work was supported by an NSERC Discovery Grant and by the Fonds de recherche du Qu\'{e}bec - Nature et technologies.

\bibliography{bibliographyVP}

\end{document}